\documentclass[
	framed_theorems, framed_proofs, print,
	english
]{OCPreprint}

\usepackage{esint}

\let\cite\parencite

\begin{document}

\title[Second-order conditions under finitely many nonlinear constraints]{No-gap second-order conditions under $n$-polyhedric constraints and finitely many nonlinear constraints}
\author{%
	Gerd Wachsmuth%
}
\maketitle

\begin{abstract}
	We consider an optimization problem subject to an abstract constraint
	and finitely many nonlinear constraints.
	Using the recently introduced concept of $n$-polyhedricity,
	we are able to provide
	second-order optimality conditions
	under weak regularity assumptions.
	In particular,
	we prove necessary optimality conditions of first and second order
	under the constraint qualification of Robinson, Zowe and Kurcyusz.
	Similarly, sufficient optimality conditions are stated.
	The gap between both conditions is as small as possible.
\end{abstract}


\begin{msc}
\mscLink{49K27}
\end{msc}

\section{Introduction}
In this work, we are interested in problems of type
\begin{equation*}
	\tag{P}
	\label{eq:prob}
	\begin{aligned}
		\text{Minimize} \quad & f(x) \\
		\text{such that} \quad & x \in C \\
		\text{and} \quad & g_i(x) = 0, \; i = 1,\ldots, m_1, \\
		& g_i(x) \le 0, \; i = m_1+1,\ldots,m.
	\end{aligned}
\end{equation*}
Our goal is the derivation of second-order necessary and sufficient optimality conditions
with minimal gap.
Here, $f \colon X \to \R$ and $g_i \colon X \to \R$, $i = 1,\ldots,m$,
are twice Fréchet differentiable,
$X$ is a Banach space
and $C \subset X$
is assumed to be closed and convex.
In particular, we interested in the situation that
\begin{equation}
	\label{eq:function_space_setting}
	X = L^2(\Omega),
	\quad
	C = \{ x \in L^2(\Omega) \mid x_a \le x \le x_b\},
\end{equation}
where $(\Omega, \Sigma, \nu)$
is a finite measure space
and
$x_a, x_b \colon \Omega \to \R \cup \{\pm\infty\}$ are measurable functions
such that $C$ becomes non-empty.

It is clear that first-order necessary optimality conditions
for \eqref{eq:prob} can be obtained
by using the constraint qualification
of Robinson-Zowe-Kurcyusz (RZKCQ),
see \cite{ZoweKurcyusz1979,Robinson1976:1}.
There are a couple of papers available
in which second-order conditions for problems of type \eqref{eq:prob}
are considered.
We mention exemplarily
\cite{BonnansZidani1999,CasasTroeltzsch2002}.
However, in these papers,
the authors have to impose additional
regularity assumptions to arrive at second-order necessary conditions.
These regularity conditions are rather strong
and, in particular, they imply uniqueness of Lagrange multipliers.
Thus,
these conditions cannot be satisfied in situations
in which the Lagrange multipliers associated
with a stationary point are not unique.
The case of infinite-dimensional equality constraints
is considered in \cite{Ioffe1979}.

In this paper,
we are utilizing
the recently introduced notion of $n$-polyhedricity,
see \cite[Definition~4.3]{Wachsmuth2016:2}
and \cref{subsec:n_polyhedricity} below,
to derive second-order necessary conditions.
Note that the set $C$ from \eqref{eq:function_space_setting}
is $n$-polyhedric for all non-negative integers $n$, see \cite[Example~4.21(1)]{Wachsmuth2016:2}.
Let $\bar x$ be a local optimizer of \eqref{eq:prob}.
Under the RZKCQ,
there exist multipliers
$\lambda \in \NN_C(\bar x)$,
$\mu \in \R^m$
such that
\begin{align*}
	\LL'(\bar x, \lambda, \mu) = 
	f'(\bar x) + \lambda + \sum_{i=1}^m \mu_i \, g_i'(\bar x) &= 0, \\
	0 \le \mu_i,
	\;
	\mu_i \, g_i(\bar x) &= 0, \quad i = m_1+1,\ldots,m
	.
\end{align*}
Here, $\LL$ denotes the Lagrangian
(see \eqref{eq:lagrangian} below)
and
$\NN_C(\bar x)$ is the normal cone of $C$ at the point $\bar x$.
The set of all multipliers satisfying the
above conditions
is denoted by $\Lambda(\bar x)$.
Our main contributions are the following.
We shall show that the condition
\begin{equation}
	\label{eq:SNC_intro}
	\sup_{(\lambda, \mu) \in \Lambda(\bar x)} \LL''(\bar x, \lambda, \mu)\,h^2
	\ge 0
	\qquad\forall h \in \TT_C(\bar x), f'(\bar x) \, h = 0
\end{equation}
is
necessary for local optimality
if RZKCQ is satisfied and if the set $C$ is $n$-polyhedric,
where $n$ is bigger than the number of active constraints.
If, additionally, a quadratic growth condition
is satisfied at $\bar x$,
we can show
\begin{equation*}
	\sup_{(\lambda, \mu) \in \Lambda(\bar x)} \LL''(\bar x, \lambda, \mu)\,h^2
	\ge \alpha \, \norm{h}_X^2
	\qquad\forall h \in \TT_C(\bar x), f'(\bar x) \, h = 0
\end{equation*}
for some $\alpha > 0$.
Under a slight additional assumption,
this last condition is also sufficient for the local optimality of $\bar x$.

The paper is organized as follows.
In \cref{sec:not} we introduce the necessary notation and review some known results.
The well-known first-order optimality conditions are given in \cref{sec:foc}.
The main results of this paper
concerning the second-order conditions for \eqref{eq:prob}
are given in \cref{sec:soc}.
In \cref{sec:examples}
we present two examples which indicate that the results in this paper are sharp.
The first example shows that the supremum in
\eqref{eq:SNC_intro}
is really necessary if the Lagrange multipliers are not unique.
The second example demonstrates
that the $n$-polyhedricity assumption on $C$ is crucial
and cannot be replaced by requiring
polyhedricity only.

\section{Notation, preliminaries and known results}
\label{sec:not}
\subsection{Notation}
We use the definitions
$\N := \{1,\ldots\}$
and
$\N_0 := \{0,1,\ldots\}$.

For a convex subset $D \subset Y$ of a Banach space $Y$
and $v \in D$,
we define
the radial cone,
the tangent cone, 
the normal cone
and the
polar cone
via
\begin{align*}
	\RR_D(v) &:= \bigcup_{\kappa > 0} \kappa (D - v),
	&
	D\polar &:= \set{\nu \in Y\dualspace \given \dual{\nu}{d} \le 0 \; \forall d \in D}
	\\
	\TT_D(v) &:= \overline{\RR_D(v)},
	&
	\NN_D(v) &:= \set{ \nu \in Y\dualspace \given \dual{\nu}{d - v} \le 0 \; \forall d \in D}
	= \TT_D(v)\polar,
\end{align*}
respectively.
The annihilator of a functional $\nu \in Y\dualspace$
is defined as
\begin{equation*}
	\nu\anni := \set{y \in Y \given \dual{\nu}{y} = 0}.
\end{equation*}
For $\varepsilon > 0$ and $y \in Y$, we define the closed ball
\begin{equation*}
	B_\varepsilon(y) := \set[\big]{\hat y \in Y \given \norm{y - \hat y}_Y \le \varepsilon}
	.
\end{equation*}

In order to discuss \eqref{eq:prob},
it will be convenient to define
$K \subset \R^m$ via
\begin{equation*}
	K
	:=
	\set[\big]{
		z \in \R^m
		\given
		z_i = 0, \; i=1,\ldots,m_1,\;
		z_i \le 0, \; i = m_1+1,\ldots,m
	}
	.
\end{equation*}
Moreover, we consider
$g = (g_1,\ldots,g_m)$
as a function from $X$ to $\R^m$.
Let a point $x \in X$ with $g(x) \in K$ be given.
Using the active and inactive sets of indices, defined via
\begin{align*}
	I_0(x) &:= \set[\big]{ i \in \{1,\ldots,m\} \given g_i(x) = 0 }
	,
	&
	I_-(x) &:= \set[\big]{ i \in \{1,\ldots,m\} \given g_i(x) < 0 }
	,
\end{align*}
respectively,
it is easy to check that
\begin{equation}
	\label{eq:R_is_T}
	\TT_K(g(x))
	=
	\RR_K(g(x))
	=
	\set*{
		z \in \R^m \given 
		\begin{aligned}
			z_i &= 0 \;\forall i=1,\ldots,m_1,
			\\
			z_i &\le 0 \; \forall i \in I_0(x) \setminus \{1,\ldots,m_1\}
		\end{aligned}
	}
	.
\end{equation}
Moreover, for a feasible point $\bar x$ of \eqref{eq:prob}
we define the critical cone $\KK(\bar x)$
via
\begin{equation}
	\label{eq:critical_cone}
	\KK(\bar x)
	=
	\set[\big]{ h \in \TT_C(\bar x) \given g'(\bar x) \, h \in \TT_K(g(\bar x)),\; f'(\bar x) \, h \le 0}
	.
\end{equation}

 
\subsection{On \texorpdfstring{$n$}{n}-polyhedricity}
\label{subsec:n_polyhedricity}
As mentioned in the introduction,
we are going to employ the concept of $n$-polyhedricity
to derive second-order conditions for \eqref{eq:prob}.
The notion of $n$-polyhedricity
was recently introduced in \cite{Wachsmuth2016:2}
and generalizes
the well-known notion of polyhedricity
due to \cite{Mignot1976,Haraux1977}.

We recall that a closed convex set $C \subset X$
is called polyhedric at $x \in C$
if
\begin{equation*}
	\cl\paren[\big]{ \RR_C(x) \cap \mu\anni}
	=
	\TT_C(x) \cap \mu\anni
	\qquad\forall \mu \in \NN_C(x).
\end{equation*}
It was shown in \cite[Lemma~4.1]{Wachsmuth2016:2}
that this condition equivalent to
\begin{equation*}
	\cl\paren[\big]{ \RR_C(x) \cap \mu\anni}
	=
	\TT_C(x) \cap \mu\anni
	\qquad\forall \mu \in X\dualspace.
\end{equation*}
The latter condition is amenable
to the following generalization.
We say that $C \subset X$
is $n$-polyhedric at $x \in C$ for some $n \in \N_0$, if
\begin{equation*}
	\TT_C(x) \cap \bigcap_{i = 1}^n \mu_i\anni
	=
	\cl\paren[\Big]{
		\RR_C(x) \cap \bigcap_{i = 1}^n \mu_i\anni
	}
	\qquad
	\forall \mu_1,\ldots,\mu_n \in X\dualspace
\end{equation*}
holds,
see \cite[Definition~4.3]{Wachsmuth2016:2}.
Many sets which were known to be polyhedric
are even $n$-polyhedric for all $n \in \N_0$,
see, e.g., \cite[Example~4.21]{Wachsmuth2016:2}.
In particular, this applies to the set of interest
$C$ from \eqref{eq:function_space_setting}.

We provide a lemma,
which follows from a simple calculation,
see also \cite[Lemma~4.4]{Wachsmuth2016:2}.
\begin{lemma}
	\label{lem:n_polyhedricity}
	Assume that the set $C \subset X$ is $N$-polyhedric
	for some $N \in \N_0$ at $\bar x \in C$.
	Further, let
	$n, n_1 \in \N_0$, $\nu_i \in X\dualspace$
	for $1\le i \le n$
	be given such that
	$N \ge n \ge n_1$.
	Then, the set
	\begin{equation*}
		\set[\Big]{
			h \in \RR_{C}(\bar x)
			\given
			\dual{\nu_i}{h} = 0,\; i = 1,\ldots, n_1,\;
			\dual{\nu_i}{h} \le 0,\; i = n_1+1,\ldots,n
		}
	\end{equation*}
	is dense in
	\begin{equation*}
		\set[\Big]{
			h \in \TT_{C}(\bar x)
			\given
			\dual{\nu_i}{h} = 0,\; i = 1,\ldots, n_1,\;
			\dual{\nu_i}{h} \le 0,\; i = n_1+1,\ldots,n
		}.
	\end{equation*}
\end{lemma}

\subsection{Review of known results}
We start by reviewing
the results of
\cite{CasasTroeltzsch2002},
see also \cite{CasasTroeltzsch1999}.
In this paper,
the authors studied a problem
very similar to \eqref{eq:prob} with
\eqref{eq:function_space_setting}.
However, they considered the situation in which
the underlying space $X$
is a Lebesgue space $L^\infty(\Omega)$
and their analysis incorporates
the important phenomenon of two-norm discrepancy.
In the situation in which all functions are already
differentiable in $X = L^2(\Omega)$,
the problem of \cite{CasasTroeltzsch2002}
coincides with \eqref{eq:prob}.
The main assumption for deriving second-order necessary conditions
is a regularity assumption on the solution $\bar x$.
For $\varepsilon > 0$,
the $\varepsilon$-inactive set is defined via
\begin{equation*}
	\Omega_\varepsilon :=
	\set{ \omega \in \Omega \given x_a(\omega) + \varepsilon \le \bar x(\omega) \le x_b(\omega) - \varepsilon}
	.
\end{equation*}
With this notation, the regularity condition is given by
\begin{equation}
	\label{eq:regularity_casas_troeltzsch}
	\exists \varepsilon > 0, \{h_j\}_{j \in I_0(\bar x)} \subset L^\infty(\Omega_\varepsilon)
	: \;
	\forall i,j \in I_0(\bar x)
	: \;
	g_i'(\bar u) \, h_j = \delta_{ij}
	.
\end{equation}
Here, we used the notation
\begin{equation*}
	L^\infty(\Omega_\varepsilon)
	=
	\set{h \in L^\infty(\Omega) \given h = 0 \text{ a.e.\ in } \Omega \setminus \Omega_\varepsilon}
	.
\end{equation*}
Under the regularity assumption \eqref{eq:regularity_casas_troeltzsch},
\cite{CasasTroeltzsch2002}
prove the existence of unique multipliers
$\mu \in \R^m$, $\lambda \in L^2(\Omega)$
such that
\begin{subequations}
	\label{eq:KKT}
	\begin{align}
		\label{eq:KKT_1}
		\mu_j &\ge 0 \qquad \forall j = m_1+1,\ldots,m,
		\\
		\label{eq:KKT_4}
		\mu_j &= 0 \qquad \forall j \in I_-(\bar x),
		\\
		\label{eq:KKT_2}
		\lambda &\in \NN_C(\bar x),
		\\
		\label{eq:KKT_3}
		f'(\bar x) + \lambda + \sum_{j = 1}^m \mu_j \, g_j'(\bar x) &= 0
		.
	\end{align}
\end{subequations}
Moreover,
they prove the second-order necessary condition
\begin{equation}
	\label{eq:SNC_casas_troeltzsch}
	f''(\bar x) \, h^2 + \sum_{j = 1}^m \mu_j \, g_j''(\bar x)\,h^2
	\ge0
	\qquad
	\forall
	h \in \KK(\bar x) \cap L^\infty(\Omega)
	.
\end{equation}
The appearance of $L^\infty(\Omega)$
in this formula comes through the general setting
of \cite{CasasTroeltzsch2002}
which includes the two-norm discrepancy.
We mention that also sufficient second-order conditions are derived.

Next, we review the results of \cite{BonnansZidani1999}.
In this work,
a problem slightly more general than \eqref{eq:prob} is considered.
In fact, the nonlinear constraints are replaced by
$G(x) \in K_Y$, where $G$ is twice Fréchet differentiable
and $K_Y$ is a closed convex set in the Banach space $Y$.
However, the strongest results are obtained in the case
that $K_Y$ is a polyhedron, i.e., a finitely intersection of
closed half-spaces,
and this is very similar to \eqref{eq:prob}.
To facilitate the comparison with our results,
we apply their results to our problem \eqref{eq:prob}.
In this case, they use the regularity condition
\begin{equation*}
	0 \in \interior
	\paren[\Big]{
		g'(\bar x) \, \bracks{ (C - \bar x) \cap \lambda\anni}
		-
		\bracks{ (K - g(\bar x)) \cap \mu\anni}
	}
\end{equation*}
for a given KKT multiplier $(\lambda, \mu)$.
Via the generalized open mapping theorem from \cite{ZoweKurcyusz1979},
this condition is equivalent to
\begin{equation}
	\label{eq:regularity_bonnans}
	\R^m
	=
	g'(\bar x) \, \bracks{ \RR_C(\bar x) \cap \lambda\anni}
	-
	\bracks{ \RR_K(\bar x) \cap \mu\anni}
	.
\end{equation}
In the literature,
this condition is often
called
``strict qualification condition''.
To our knowledge,
this condition appears first in \cite[Theorem~3.3]{MaurerZowe1979}.
Moreover, it is known
that this condition implies the uniqueness of the multipliers $(\lambda,\mu)$,
see \cite{Shapiro1997}.
Moreover, it is straightforward to check
that \eqref{eq:regularity_casas_troeltzsch} is strictly stronger
than \eqref{eq:regularity_bonnans}.
Under condition \eqref{eq:regularity_bonnans},
\cite[Theorem~2.7(iii)]{BonnansZidani1999}
gives
the second-order necessary condition
\begin{equation}
	\label{eq:SNC_bonnans_zidani}
	f''(\bar x) \, h^2 + \sum_{j = 1}^m \mu_j \, g_j''(\bar x)\,h^2
	\ge0
	\qquad
	\forall
	h \in \KK(\bar x)
	.
\end{equation}
Under the additional assumption that the second derivative
of the Lagrangian is a Legendre form,
they also derive sufficient conditions.
Consequently, the gap between necessary and sufficient conditions
of second order is as small as possible.

Using the inheritance property \cite[Lemma~3.3]{Wachsmuth2016:2}
of polyhedric sets,
it is possible to generalize the results of
\cite{BonnansZidani1999}
in the following way.
Instead of \eqref{eq:prob},
we consider
the much more general problem
\begin{equation*}
	\begin{aligned}
		\text{Minimize} \quad & f(x) \\
		\text{such that} \quad & x \in C \\
		\text{and} \quad & G(x) \in D.
	\end{aligned}
\end{equation*}
Here, $X$, $Y$ are a Banach spaces,
$f \colon X \to Y$ $G \colon X \to Y$ are
twice Fréchet differentiable
and $C \subset X$, $D \subset Y$ are closed, convex and polyhedric sets.
Given multipliers
$(\lambda,\mu) \in \NN_C(\bar x) \times \NN_D(G(\bar x))$,
the condition \eqref{eq:regularity_bonnans} becomes
\begin{equation}
	\label{eq:sRZKC}
	Y
	=
	G'(\bar x) \, \bracks{ \RR_C(\bar x) \cap \lambda\anni}
	-
	\bracks{ \RR_D(\bar x) \cap \mu\anni}
	.
\end{equation}
Under this condition, we can apply
\cite[Theorem~5.4]{Wachsmuth2016:2}
and obtain the second-order necessary condition
\begin{equation}
	\label{eq:SNC_polyhedric}
	f''(\bar x) \, h^2 + \dual{\mu}{G''(\bar x) \, h^2}
	\ge0
	\qquad
	\forall
	h \in \TT_C(\bar x),
	G'(\bar x) \, h \in \TT_D(G(\bar x)),
	f'(\bar x) \, h = 0
	.
\end{equation}
Note that one has to rewrite the constraints as
$\hat G(x) := (x, G(x)) \in C \times D =: \hat K$
to apply this theorem.
Thus,
if this strong regularity condition \eqref{eq:sRZKC} is satisfied,
we can replace the assumption of $K$
being polyhedral in
\cite{BonnansZidani1999}
by the much weaker assumption
of polyhedricity.
We note that also necessary conditions of second order can be found in
\cite[Theorems~5.6, 5.7]{Wachsmuth2016:2}.

\section{First-order optimality conditions and constraint qualifications}
\label{sec:foc}
In this section, we briefly recall first-order optimality conditions
for the problem \eqref{eq:prob}
and the constraint qualifications which are required for the derivation.
In order to put our problem into the framework of \cite{ZoweKurcyusz1979},
we recall
\begin{equation*}
	K
	=
	\{
		z \in \R^m \mid
		z_i = 0, \; i=1,\ldots,m_1,\;
		z_i \le 0, \; i = m_1+1,\ldots,m
	\}
\end{equation*}
and
$g = (g_1,\ldots,g_m)$.
Now, our problem \eqref{eq:prob}
reads
\begin{equation}
	\label{eq:prob2}
	\text{Minimize } f(x)
	\quad
	\text{subject to $x \in C$ and $g(x) \in K$}.
\end{equation}
An application of
\cite[Theorem~3.1]{ZoweKurcyusz1979}
implies the following first-order necessary conditions.
\begin{theorem}
	\label{thm:fonc}
	Assume that $\bar x \in X$ is a local minimizer of \eqref{eq:prob}
	such that
	\begin{equation}
		\label{eq:rzcq}
		\tag{RZKCQ}
		g'(\bar x) \, \RR_C(\bar x)
		- \RR_K(g(\bar x))
		=
		\R^m
	\end{equation}
	is satisfied.
	Then, there exist
	$\lambda \in \NN_C(\bar x)$,
	$\mu \in \NN_K(g(\bar x))$
	such that
	\begin{equation}
		\label{eq:fonc}
		f'(\bar x) + \lambda + g'(\bar x)\adjoint \mu = 0.
	\end{equation}
\end{theorem}
It is clear that
$\mu \in \NN_K(g(\bar x))$
is equivalent to
\begin{equation*}
	\mu_i \le 0
	\quad\text{and}\quad
	\mu_i \, g_i(\bar x) = 0
	\qquad\forall i = m_1+1,\ldots, m
	.
\end{equation*}
Further, condition \eqref{eq:fonc} can be written concisely as
\begin{equation*}
	\LL'(\bar x, \lambda, \mu) = 0,
\end{equation*}
where
the Lagrangian $\LL \colon X \times X\dualspace \times \R^m \to \R$
is defined via
\begin{equation}
	\label{eq:lagrangian}
	\LL(x, \lambda, \mu)
	=
	f(x) + \dual{\lambda}{x} + \sum_{i = 1}^m \mu_i \, g_i(x)
\end{equation}
and
a prime denotes partial differentiation w.r.t.\ $x$.
For convenience, we recall the expressions for the first and second derivative of
the Lagrangian $\LL$
w.r.t.\ $x$
\begin{align*}
	\LL'(x, \lambda, \mu)
	&=
	f'(x)
	+
	\lambda
	+
	\sum_{i=1}^m \mu_i \, g_i'(x)
	\\
	\LL''(x, \lambda, \mu) \, h^2
	&=
	f''(x) \, h^2
	+
	\sum_{i=1}^m \mu_i \, g_i''(x) \, h^2
\end{align*}
for $h \in X$.
Here, we used the
common abbreviation $h^2$ for the action of a bilinear form
on the tuple $[h,h]$.

For an arbitrary feasible point $x$, we define the set of Lagrange multipliers via
\begin{equation}
	\label{eq:mult}
	\Lambda(x)
	:=
	\set[\big]{
		(\lambda, \mu) \in \NN_C(x) \times \NN_K(g(x))
		\given
		\LL'(x, \lambda, \mu) = 0
	}
	.
\end{equation}
We also recall from
\cite[Theorem~4.1]{ZoweKurcyusz1979}
that \eqref{eq:rzcq}
implies the boundedness of $\Lambda(\bar x)$.
We mentioned that the boundedness of $\Lambda(\bar x)$
can be shown under the slightly weaker condition
\begin{equation}
	\label{eq:weaker_rzcq}
	g'(\bar x) \, \TT_C(\bar x)
	- \TT_K(g(\bar x))
	=
	\R^m
\end{equation}
by a suitable modification of
the proof of
\cite[Theorem~4.1]{ZoweKurcyusz1979}.
Note that, however, $\Lambda(\bar x)$
might be empty if only \eqref{eq:weaker_rzcq} is satisfied.

\section{No-gap second-order optimality conditions}
\label{sec:soc}
In this section, we consider
second-order optimality conditions for problem \eqref{eq:prob}.

We begin by the derivation of necessary optimality conditions.
In order to apply the results from
\cite[Section~3.2.3]{BonnansShapiro2000},
we introduce
\begin{equation*}
	G(x) := (x, g(x)),
	\quad
	\hat K := C \times K.
\end{equation*}
Now,
\eqref{eq:prob}
reads
\begin{equation*}
	\begin{aligned}
		\text{Minimize}\quad & f(x) \\
		\text{such that} \quad & G(x) \in \hat K.
	\end{aligned}
\end{equation*}
Let $\bar x$ be a feasible point of \eqref{eq:prob}.
From
\cite[(3.20) and (3.122)]{BonnansShapiro2000},
we recall the definition of
the critical cone
\begin{equation}
	\label{eq:critical}
	\begin{aligned}
		\KK(\bar x)
		&:=
		\set[\big]{ h \in X \given G'(\bar x) \, h \in \TT_{\hat K}(G(\bar x)), \; f'(\bar x) \, h \le 0 }
		\\
		&=
		\set[\big]{ h \in \TT_C(\bar x) \given g'(\bar x) \, h \in \TT_K(g(\bar x)), \; f'(\bar x) \, h \le 0 }
	\end{aligned},
\end{equation}
which matches our definition \eqref{eq:critical_cone},
and of the set of radial critical directions
\begin{align*}
	\KK_R(\bar x)
	&:=
	\set[\big]{ h \in X \given G'(\bar x) \, h \in \RR_{\hat K}(G(\bar x)), \; f'(\bar x) \, h \le 0 }
	\\
	&=
	\set[\big]{ h \in \RR_C(\bar x) \given g'(\bar x) \, h \in \RR_K(g(\bar x)), \; f'(\bar x) \, h \le 0 }.
\end{align*}
Note that
we have
\begin{equation*}
	\KK_R(\bar x)
	=
	\set[\big]{ h \in \RR_C(\bar x) \given g'(\bar x) \, h \in \TT_K(g(\bar x)), \; f'(\bar x) \, h \le 0 }
\end{equation*}
due to \eqref{eq:R_is_T}.

From \cite[Proposition~3.53]{BonnansShapiro2000}
we get the following result.
\begin{lemma}
	\label{lem:necessary_condition}
	Assume that $\bar x$ is a local minimizer of \eqref{eq:prob}
	such that \eqref{eq:rzcq} is satisfied.
	Further suppose that
	$\KK_R(\bar x)$ is dense in $\KK(\bar x)$.
	Then,
	\begin{equation*}
		\sup_{(\lambda, \mu) \in \Lambda(\bar x)} \LL''(\bar x, \lambda, \mu) \, h^2 \ge 0
		\qquad
		\forall h \in \KK(\bar x).
	\end{equation*}
\end{lemma}
The density assumption in this result
can be shown
under an additional condition on the constraint set $C$.
\begin{theorem}
	\label{thm:SNC}
	Assume that $\bar x$ is a local minimizer of \eqref{eq:prob}
	such that \eqref{eq:rzcq} is satisfied.
	We denote by $\hat m$ the number of active constraints in $\bar x$,
	i.e., the number of indices $i = 1,\ldots,m$ with $g_i(\bar x) = 0$.
	Under the assumption that
	$C$
	is $(\hat m + 1)$-polyhedric,
	we have
	\begin{equation}
		\label{eq:snc}
		\sup_{(\lambda, \mu) \in \Lambda(\bar x)} \LL''(\bar x, \lambda, \mu) \, h^2 \ge 0
		\qquad
		\forall h \in \KK(\bar x).
	\end{equation}
\end{theorem}
\begin{proof}
	We recall the formula
	\begin{equation*}
		\TT_K(g(\bar x))
		=
		\set*{
			z \in \R^m \given 
				z_i = 0 \;\forall i=1,\ldots,m_1,\;
				z_i \le 0 \; \forall i \in I_0(\bar x) \setminus \{1,\ldots,m_1\}
		}
		.
	\end{equation*}
	for the tangent cone of $K$,
	where
	$I_0(\bar x)$
	denotes the set of active indices.
	For brevity, we set
	$\hat I_0(\bar x) := I_0(\bar x) \setminus \{1,\ldots,m_1\}$.
	Thus,
	\begin{align*}
		\KK(\bar x)
		&=
		\set{
			h \in \TT_C(\bar x)
			\given
			g_i'(\bar x) \, h = 0, i = 1,\ldots,m_1, \;
			g_i'(\bar x) \, h \le 0, i \in \hat I_0(\bar x), \;
			f'(\bar x) \, h \le 0
		}
		,
		\\
		\KK_R(\bar x)
		&=
		\set{
			h \in \RR_C(\bar x)
			\given
			g_i'(\bar x) \, h = 0, i = 1,\ldots,m_1, \;
			g_i'(\bar x) \, h \le 0, i \in \hat I_0(\bar x), \;
			f'(\bar x) \, h \le 0
		}
		.
	\end{align*}
	In these sets, we have $\hat m+1$ many scalar equalities and inequalities.
	Due to the assumption that $C$ is $(\hat m + 1)$-polyhedric,
	we can invoke
	\cref{lem:n_polyhedricity}.
	This implies that $\KK_R(\bar x)$ is dense in $\KK(\bar x)$.
	Thus, the assertion follows from \cref{lem:necessary_condition}.
\end{proof}
Note that the supremum
in the above inequality is attained,
since the set of multipliers is weak-$\star$
compact, see \cite[Theorem~3.9]{BonnansShapiro2000},
and the second derivative of the Lagrangian
is weak-$\star$ continuous w.r.t.\ the multipliers.
Hence, \eqref{eq:snc} can be rephrased as follows.
For every critical direction $h \in \KK(\bar x)$,
there exist multipliers $(\lambda, \mu) \in \Lambda(\bar x)$
such that $\LL''(\bar x, \lambda, \mu) \, h^2 \ge 0$.

If a quadratic growth condition is satisfied at $\bar x$,
we get a better inequality.
\begin{corollary}
	\label{cor:SNC_growth}
	Additionally to the assumptions of \cref{thm:SNC},
	we assume that the growth condition
	\begin{equation}
		\label{eq:growth}
		f(x)
		\ge
		f(\bar x)
		+
		\frac\alpha2\,\norm{x - \bar x}_X^2
		\qquad
		\forall x \in F \cap B_\varepsilon(\bar x)
	\end{equation}
	is satisfied for some $\alpha,\varepsilon > 0$ at $\bar x$,
	where
	$F = \set{x \in C \given g(x) \in K}$
	is the feasible set of \eqref{eq:prob}.
	Then,
	\begin{equation*}
		\sup_{(\lambda, \mu) \in \Lambda(\bar x)} \LL''(\bar x, \lambda, \mu) \, h^2 \ge
		\alpha\,\norm{h}_X^2
		\qquad
		\forall h \in \KK(\bar x).
	\end{equation*}
\end{corollary}
\begin{proof}
	Under \eqref{eq:growth}, $\bar x$
	is a local minimizer of
	$\hat f(x) := f(x) - \frac\alpha2 \, \norm{x - \bar x}_X^2$
	on $F$.
	Note that $\hat f$ is twice Fréchet differentiable if $X$ is a Hilbert space.
	In this case, a direct
	application of \cref{thm:SNC}
	yields the claim.
	If $X$ is not a Hilbert space,
	we can still reproduce
	\cite[Lemma~3.44]{BonnansShapiro2000},
	which is enough to prove
	\cite[Prop.~3.53]{BonnansShapiro2000}
	and, consequently, \cref{thm:SNC}.
	To this end, we set $\tilde f(x) := \frac12 \, \norm{x - \bar x}_X^2$
	and check that a second-order Taylor expansion
	similar to
	\cite[(3.100)]{BonnansShapiro2000}
	holds.
	To this end, let $h,w \in X$ and $r : (0,\infty) \to X$
	be given such that $r(t) = \oo(t^2)$.
	We define the path $x(t) := \bar x + t \, h + \frac12 \, t^2 \, w + r(t)$.
	Then,
	\begin{align*}
		\abs[\big]{
			\tilde f(x(t)) - \tilde f(\bar x) - \frac12 \, t^2 \, \norm{h}_X^2
		}
		&=
		\frac12\,\abs[\bigg]{
			\norm[\Big]{t \, h + \frac12 \, t^2 \, w + r(t)}_X^2 - \norm{t \, h}_X^2
		}
		\\
		&\le
		C \, t
		\,
		\abs[\bigg]{
			\norm[\Big]{t \, h + \frac12 \, t^2 \, w + r(t)}_X - \norm{t \, h}_X
		}
		\\
		&\le
		C \, t
		\,
		\abs[\bigg]{
			\norm[\Big]{\frac12 \, t^2 \, w + r(t)}_X
		}
		= \oo(t^2)
	\end{align*}
	as $t \searrow 0$.
	Hence, the modified function $\tilde f$ satisfies
	the required second-order Taylor expansion.
\end{proof}

As usual, second-order sufficient conditions
can be derived by a contradiction argument.
\begin{theorem}
	\label{thm:sufficient_condition}
	Assume that $\bar x$ is a stationary point of \eqref{eq:prob}, i.e.,
	there exist $(\lambda, \mu) \in \Lambda(\bar x)$.
	Further, we suppose that the CQ
	\eqref{eq:weaker_rzcq},
	which is slightly weaker than Robinson's CQ,
	be satisfied.
	We assume that
	\begin{equation}
		\label{eq:SSC}
		\sup_{(\lambda, \mu) \in \Lambda(\bar x)} \LL''(\bar x, \lambda, \mu) \, h^2
		\ge
		\alpha \, \norm{h}_X^2
		\qquad
		\forall h \in \KK_\eta(\bar x)
	\end{equation}
	holds for some $\alpha, \eta > 0$,
	where the extended critical cone $\KK_\eta(\bar x)$ is given by
	\begin{equation*}
		\KK_\eta(\bar x)
		:=
		\set[\Big]{
			h \in \TT_C(\bar x)
			\given
			g'(\bar x) \, h \in \TT_K(g(\bar x))
			\;\text{and}\;
			f'(\bar x) \, h \le \eta \, \norm{h}_X
		}.
	\end{equation*}
	Then, for all $\tilde\alpha \in (0, \alpha)$, there is $\varepsilon > 0$
	such that
	\begin{equation*}
		f(x) \ge f(\bar x) + \frac{\tilde\alpha}{2} \, \norm{x - \bar x}_X^2
		\qquad\forall x \in F \cap B_\varepsilon(\bar x),
	\end{equation*}
	where
	$F = \set{x \in C \given g(x) \in K }$
	is the feasible set of \eqref{eq:prob}.
\end{theorem}
\begin{proof}
	We fix $\tilde\alpha \in (0,\alpha)$ and proceed by contradiction.
	This yields a sequence $x_n \in F \setminus\{\bar x\}$ with $x_n \to \bar x$
	and $f(x_n) < f(\bar x) + \frac{\tilde\alpha}2 \, \norm{x_n - \bar x}_X^2$.

	Using the Fréchet differentiability of $g$, we have
	\begin{equation*}
		r_n := g(\bar x) + g'(\bar x) \, (x_n - \bar x) - g(x_n)
		=
		\oo(\norm{x_n - \bar x}_X).
	\end{equation*}
	Owing to the CQ and the generalized open mapping theorem
	\cite[Theorem~2.1]{ZoweKurcyusz1979},
	we find sequences $\{h_n\} \subset \TT_C(\bar x)$,
	$\{v_n\} \subset \TT_K(g(\bar x))$
	with
	\begin{equation*}
		-r_n = g'(\bar x) \, h_n - v_n
	\end{equation*}
	and $h_n = \OO(\norm{r_n}_X) = \oo(\norm{x_n - \bar x}_X)$.
	In particular,
	$x_n - \bar x + h_n \in \TT_C(\bar x)$
	and
	\begin{equation*}
		g'(\bar x) \, (x_n - \bar x + h_n) = g(x_n) - g(\bar x) + v_n \in \TT_K(g(\bar x)).
	\end{equation*}
	Further,
	\begin{align*}
		f'(\bar x) \, (x_n - \bar x)
		&=
		f(x_n) - f(\bar x) + \oo(\norm{x_n - \bar x}_X)
		\\
		&\le
		\frac{\tilde\alpha}{2} \, \norm{x_n - \bar x}_X^2 + \oo(\norm{x_n - \bar x}_X)
		=
		\oo(\norm{x_n - \bar x}_X)
	\end{align*}
	yields $f'(\bar x) \, (x_n - \bar x + h_n) = \oo(\norm{x_n - \bar x}_X) = \oo(\norm{x_n - \bar x + h_n}_X)$.
	Hence, $x_n - \bar x + h_n \in \KK_\eta(\bar x)$ for $n$ large enough.

	Now, for large $n$, we choose $(\lambda_n, \mu_n) \in \Lambda(\bar x)$, such that
	\begin{equation*}
		\LL''(\bar x, \lambda_n, \mu_n) \, (x_n - \bar x + h_n)^2
		\ge
		\paren[\Big]{ \alpha  - \frac{\alpha - \tilde\alpha}{2}}\, \norm{x_n - \bar x + h_n}_X^2
		=
		\frac{\alpha+\tilde\alpha}2
		\, \norm{x_n - \bar x + h_n}_X^2
	\end{equation*}
	This is possible since $x_n - \bar x + h_n \ne 0$ for $n$ large enough.

	For $n$ large enough we have
	\begin{align*}
		\frac{\tilde\alpha}{2}  \, \norm{x_n - \bar x}^2
		&>
		f(x_n) - f(\bar x)
		\\&\ge
		f(x_n) - f(\bar x) + \dual{\mu_n}{x_n - \bar x} + \dual{\lambda_n}{g(x_n) - g(\bar x)}
		\\&=
		\LL(x_n, \lambda_n, \mu_n) - \LL(\bar x, \lambda_n, \mu_n)
		.
	\end{align*}
	Next, we are going to use a Taylor expansion of the Lagrangian.
	Since $f$ and $g$ are twice Fréchet differentiable,
	we have the Taylor expansion
	\begin{equation*}
		f(x_n) - f(\bar x) = f'(\bar x) \, (x_n-\bar x) + \frac12 \, f''(\bar x) \, (x_n - \bar x)^2 + \oo(\norm{x_n - \bar x}^2_X)
	\end{equation*}
	and analogously for $g$.
	Now, we utilize that the CQ \eqref{eq:weaker_rzcq} implies the boundedness
	of the multipliers $\Lambda(\bar x)$.
	This yields that we can use a Taylor expansion for
	$\LL(\cdot, \lambda_n, \mu_n)$ at $\bar x$
	and the remainder term is uniform w.r.t.\ the multipliers $(\lambda_n,\mu_n) \in \Lambda(\bar x)$.
	Thus, we can continue with
	\begin{align*}
		\frac{\tilde\alpha}{2}  \, \norm{x_n - \bar x}^2
		&>
		\LL(x_n, \lambda_n, \mu_n) - \LL(\bar x, \lambda_n, \mu_n)
		\\&=
		\LL'(\bar x, \lambda_n, \mu_n) \, (x_n - \bar x)
		+
		\frac12 \, \LL''(\bar x, \lambda_n, \mu_n) \, (x_n - \bar x)^2
		+
		\oo(\norm{x_n - \bar x}_X^2)
		\\&=
		\frac12 \, \LL''(\bar x, \lambda_n, \mu_n) \, (x_n - \bar x + h_n)^2
		-
		\LL''(\bar x, \lambda_n, \mu_n) \, [x_n - \bar x, h_n]
		\\&\qquad
		-
		\frac12 \, \LL''(\bar x, \lambda_n, \mu_n) \, h_n^2
		+
		\oo(\norm{x_n - \bar x}_X^2)
	\end{align*}
	In order to deal with the second and third addend,
	we
	use again the boundedness of $\Lambda(\bar x)$.
	Together with
	$\norm{h_n}_X = \oo(\norm{x_n - \bar x}_X)$,
	both addends belong to $\oo(\norm{x_n-\bar x}_X^2)$ as $n \to \infty$.
	Thus, we can continue via
	\begin{align*}
		\frac{\tilde\alpha}{2}  \, \norm{x_n - \bar x}^2
		&>
		\frac12 \, \LL''(\bar x, \lambda_n, \mu_n) \, (x_n - \bar x + h_n)^2
		+
		\oo(\norm{x_n - \bar x}_X^2)
		\\&\ge
		\frac{\alpha+\tilde\alpha}4 \, \norm{x_n - \bar x + h_n}^2
		+
		\oo(\norm{x_n - \bar x}_X^2)
		.
	\end{align*}
	Dividing by $\norm{x_n - \bar x}_X^2$ and passing to the limit $n \to \infty$
	yields
	the contradiction
	$\tilde\alpha /2 \ge (\alpha+\tilde\alpha)/4$.
\end{proof}
Using the notion of Legendre forms,
it possible to weaken the assumed inequality \eqref{eq:SSC}.
We recall from
\cite[Section~6.2]{IoffeTichomirov1979:2}
that a continuous bilinear form $a : H \times H \to \R$
on a Hilbert space $H$ is called a Legendre form,
if $x \mapsto a(x,x)$ is sequentially weakly lower semicontinuous
and if
\begin{equation*}
	x_n \weakly x
	\quad\text{and}\quad
	a(x_n,x_n) \to a(x,x)
	\qquad\Longrightarrow\qquad
	x_n \to x
	.
\end{equation*}
Clearly, this definition can also be used if $H$
is not a Hilbert space, but only a Banach space.
However, it was shown recently in \cite{Harder2018}
that a reflexive Banach space permits a Legendre form
only if it possesses an equivalent Hilbert space norm.
The notion of Legendre forms was generalized to non-quadratic forms
in \cite[Definition~3.73]{BonnansShapiro2000}.
Therein, a function $q : X \to \R$ is called an
extended Legendre form, if it is weakly lower semicontinuous,
positively homogeneous of degree $2$ and
if
\begin{equation*}
	x_n \weakly x
	\quad\text{and}\quad
	q(x_n) \to q(x)
	\qquad\Longrightarrow\qquad
	x_n \to x
\end{equation*}
is satisfied.
We are interested in the case that
\begin{equation}
	\label{eq:def_q}
	q(h) :=
	\sup_{(\lambda, \mu) \in \Lambda(\bar x)} \LL''(\bar x, \lambda, \mu) \, h^2
	=
	f''(\bar x) \, h^2 +
	\sup_{(\lambda, \mu) \in \Lambda(\bar x)} \sum_{i=1}^m \mu_i \, g''(\bar x) \, h^2
\end{equation}
is the maximized Hessian of the Lagrangian.
Under the assumption that the set of multipliers $\Lambda(\bar x)$
is bounded, which holds, e.g., under \eqref{eq:weaker_rzcq},
and non-empty,
the function $q$ is finite, i.e., it maps $X$ to $\R$.
The next results states necessary conditions
which ensure that a sum of two functions is an extended Legendre form.
It is inspired by \cite[Proposition~3.76~(ii)]{BonnansShapiro2000}.
\begin{lemma}
	\label{lem:ex_leg_form}
	Suppose that $q_1 : X \to \R$ is an extended Legendre form
	and that $q_2 : X \to \R$ is positively homogeneous of degree $2$
	and weakly lower semicontinuous.
	Then, $q := q_1 + q_2$ is an extended Legendre form.
\end{lemma}
\begin{proof}
	It is clear that $q$ is positively homogeneous of degree $2$
	and weakly lower semicontinuous.
	Now, suppose that $x_n \weakly x$ and $q(x_n) \to q(x)$.
	From
	\begin{align*}
		q(x)
		&=
		\lim_{n \to \infty} q(x_n)
		=
		\limsup_{n \to \infty} \paren[\big]{ q_1(x_n) + q_2(x_n) }
		\ge
		\limsup_{n \to \infty} q_1(x_n) + \liminf_{n \to \infty} q_2(x_n)
		\\
		&\ge
		\liminf_{n \to \infty} q_1(x_n) + \liminf_{n \to \infty} q_2(x_n)
		\ge
		q_1(x) + q_2(x)
		=
		q(x)
	\end{align*}
	we infer $q_1(x_n) \to q_1(x)$.
	Since $q_1$ is an extended Legendre form,
	$x_n \to x$ follows.
	This shows that $q$ is an extended Legendre form.
\end{proof}
The next result is an adaption of \cite[Proposition~3.77]{BonnansShapiro2000}
to the situation at hand.
\begin{lemma}
	\label{lem:ext_leg_2}
	Let $\bar x$ be a feasible point
	such that
	$\Lambda(\bar x)$ is not empty and bounded.
	Further, we assume that $f''(\bar x)$ is a Legendre form and that
	\begin{itemize}
		\item
			$h \mapsto g_i''(\bar x) \, h^2$ is weakly continuous for all $i = 1,\ldots, m_1$ and
		\item
			$h \mapsto g_i''(\bar x) \, h^2$ is weakly lower semicontinuous for all $i \in I_0(\bar x) \setminus \{1,\ldots,m_1\}$
			.
	\end{itemize}
	Then,
	the function $q$ defined in \eqref{eq:def_q}
	is an extended Legendre form.
\end{lemma}
\begin{proof}
	For every $(\lambda,\mu) \in \Lambda(\bar x)$,
	the function $h \mapsto q_2^{(\mu)}(h) := \sum_{i = 1}^m \mu_i \, g_i''(\bar x)\,h^2$
	is weakly lower semicontinuous, since $\mu_i \ge 0$ for $i \in I_0(\bar x) \setminus \{1,\ldots,m_1\}$.
	Moreover, these functions are positively $2$-homogeneous.
	As the supremum of weakly lower semicontinuous functions,
	the function
	\begin{equation*}
		h \mapsto q_2(h) := \sup_{(\lambda,\mu) \in \Lambda(\bar x)}\sum_{i = 1}^m \mu_i \, g_i''(\bar x)\,h^2
	\end{equation*}
	is weakly lower semicontinuous.
	Now, an application of \cref{lem:ex_leg_form} yields the assertion.
\end{proof}

From \cite[Lemma~3.75]{BonnansShapiro2000},
we obtain the following result.
\begin{lemma}
	\label{lem:leg_form_coerc}
	Let $X$ be a reflexive Banach space.
	Suppose that \eqref{eq:weaker_rzcq} is satisfied at the feasible point $\bar x$
	and that $\Lambda(\bar x)$ is not empty.
	We further assume that
	\begin{equation*}
		h \mapsto \sup_{(\lambda, \mu) \in \Lambda(\bar x)} \LL''(\bar x, \lambda, \mu) \, h^2
	\end{equation*}
	is an extended Lagrange form.
	Then, the condition \eqref{eq:SSC} is equivalent to
	\begin{equation}
		\label{eq:SSC_positive}
		\sup_{(\lambda, \mu) \in \Lambda(\bar x)} \LL''(\bar x, \lambda, \mu) \, h^2
		>
		0
		\qquad
		\forall h \in \KK_0(\bar x) = \KK(\bar x).
	\end{equation}
\end{lemma}
In this case, we have a minimal gap between the necessary and sufficient conditions
of \cref{thm:SNC,thm:sufficient_condition}.

\section{Examples}
\label{sec:examples}

In this section,
we provide two examples.
These examples illustrate
two crucial ingredients of \cref{thm:SNC}.

The first example is constructed in such a way that
the assumptions of
\cref{thm:SNC} are satisfied
and, hence, the necessary conditions \eqref{eq:snc}
hold.
However,
the set of multipliers $\Lambda(\bar x)$ is not a singleton
and
the condition
\begin{equation*}
	\LL''(\bar x, \lambda, \mu) \, h^2 \ge 0
	\qquad
	\forall h \in \KK(\bar x)
\end{equation*}
is violated for all $(\lambda, \mu) \in \Lambda(\bar x)$.
Hence, it is crucial to take the supremum over all multipliers in \eqref{eq:snc}.

In the other example,
we demonstrate that the assumption that
$C$ is $(\hat m+1)$-polyhedric
is crucial.
To this end, we have to use
a polyhedric set
which is not $2$-polyhedric.

\subsection{Non-unique multipliers}
This example is heavily inspired by
\cite[Counterexample~1.2]{CrouzeixMartinezLegazSeeger1995}.
We repeat this counterexample,
since it will be important in the sequel.
We define the matrices
\begin{equation*}
	A_1 :=
	\begin{pmatrix}
		1 & 1 \\ 1 & -1
	\end{pmatrix}
	,
	\qquad
	A_2 :=
	\begin{pmatrix}
		-2 & 1 \\ 1 & 1
	\end{pmatrix}
	.
\end{equation*}
These matrices have the property that
\begin{equation}
	\label{eq:crazy_matrix_property}
	\max\{ x^\top A_1 \, x, x^\top A_2 \, x\} \ge \frac12 \, \norm{x}^2_{\R^2}
	\qquad \forall x \in \R^2 , x \ge 0
	.
\end{equation}
Indeed, this can be shown, e.g., by a distinction of the cases
$x_1 \ge x_2$ and $x_2 \ge x_1$.
However, for every $\lambda \in [0,1]$,
the convex combination
\begin{equation*}
	B_\lambda := \lambda \, A_1 + (1-\lambda) \, A_2
	=
	\begin{pmatrix}
		-2 + 3 \, \lambda & 1 \\
		1 & 1 - 2 \, \lambda
	\end{pmatrix}
\end{equation*}
is not coercive on non-negative vectors, since
at least one of the numbers
\begin{equation*}
	e_1^\top B_\lambda \, e_1 = -2 + 3 \,\lambda,
	\qquad
	e_2^\top B_\lambda \, e_2 = 1 - 2 \,\lambda
\end{equation*}
will be negative.

We are going to construct a problem of the form
\begin{equation}
	\label{eq:example1}
	\begin{aligned}
		\text{Minimize}\quad & f(x) \\
		\text{such that}\quad & x \in C, \\
		\text{and}\quad & g(x) = 0.
	\end{aligned}
\end{equation}
Here,
$f, g \colon L^2(0,1) \to \R$
are (continuous) quadratic functions
to be defined below and
\begin{equation*}
	C = \set{ u \in L^2(0,1) \given -1 \le u \le 1 \; \text{a.e.\ on } I}.
\end{equation*}
Our point of interest will be $\bar x \in C$ defined via
\begin{equation*}
	\bar x(t)
	=
	\begin{cases}
		-1 & \text{if } t \in (0,1/3), \\
		-3 + 6 \, t & \text{if } t \in (1/3,2/3), \\
		1 & \text{if } t \in (2/3,1).
	\end{cases}
\end{equation*}
It is clear that
\begin{align*}
	\TT_{C}(\bar x)
	&=
	\set[\big]{ v \in L^2(0,1) \given v \ge 0 \text{ a.e.\ on } (0,1/3) \text{ and } v \le 0 \text{ a.e.\ on } (2/3,1) }
	,
	\\
	\NN_{C}(\bar x)
	&=
	\set*{ v \in L^2(0,1) \given
		\begin{aligned}
			&v \le 0 \text{ a.e.\ on } (0,1/3),\; v = 0 \text{ a.e.\ on } (1/3,2/3) \text{ and } \\
			&v \ge 0 \text{ a.e.\ on } (2/3,1)
		\end{aligned}
	}
	.
\end{align*}
The function $g$ will satisfy
\begin{equation*}
	g(\bar x) = 0,
	\qquad
	g'(\bar x) = \chi_{(0,1/4)} + \chi_{(2/3,3/4)}.
\end{equation*}
The first conditions renders $\bar x$ feasible for \eqref{eq:example1}.
Due to
\begin{equation*}
	\RR_C(\bar x)
	\supset
	\set*{ v \in L^\infty(0,1) \given
		\begin{aligned}
			&v \ge 0 \text{ a.e.\ on } (0,1/3),\; v = 0 \text{ a.e.\ on } (1/3,2/3) \text{ and } \\
			&v \le 0 \text{ a.e.\ on } (2/3,1)
		\end{aligned}
	}
\end{equation*}
it is easy to check that
\begin{equation*}
	g'(\bar x) \, \RR_C(\bar x) = \R,
\end{equation*}
thus, \eqref{eq:rzcq} is satisfied.
Next, we require
\begin{equation*}
	f'(\bar x)
	=
	-\chi_{(2/3, 3/4)}
\end{equation*}
and we compute the set of Lagrange multipliers $\Lambda(\bar x)$.
This amounts to find all $\mu \in \R$, such that the corresponding $\lambda$
satisfies
\begin{equation*}
	\lambda = -f'(\bar x) - \mu \, g'(\bar x)
	\in
	\NN_C(\bar x)
	.
\end{equation*}
By using the formula for the normal cone,
we see that this is equivalent to $\mu \in [0,1]$.
Thus $\bar x$ is a stationary point
and
\begin{equation*}
	\Lambda(\bar x)
	=
	\set{ (-f'(\bar x) - \mu \, g'(\bar x), \mu) \given \mu \in [0,1] }
	.
\end{equation*}
Note that the critical cone
$\KK(\bar x) = \TT_C(\bar x) \cap f'(\bar x)\anni$
is given by
\begin{equation*}
	\KK(\bar x)
	=
	\set*{ v \in L^2(0,1) \given
		\begin{aligned}
			&v \ge 0 \text{ a.e.\ on } (0,1/3),\; v = 0 \text{ a.e.\ on } (2/3,3/4) \text{ and } \\
			&v \le 0 \text{ a.e.\ on } (3/4,1)
		\end{aligned}
	}
	.
\end{equation*}
Next,
we define the second derivatives
of $f$ and $g$ at $\bar x$.
To this end, we use the notation
\begin{equation*}
	\fint_a^b x \, \dt
	:=
	\frac1{b-a} \, \int_a^b x \, \dt,
	\qquad
	Z_a^b(x) := x - \fint_a^b x \, \dt
\end{equation*}
for the average of a function $x$ over an interval $(a,b)$
and for the difference of the function with this average.
With this notation,
we introduce
\begin{align*}
	f''(\bar x)[h_1,h_2]
	&:=
	\int_0^{1/3} Z_0^{1/3}(h_1) \, Z_0^{1/3}(h_2) \, \dt
	+
	\int_{1/3}^{3/4} h_1 \, h_2 \, \dt
	+
	\int_{3/4}^1 Z_{3/4}^1(h_1) \, Z_{3/4}^1(h_2) \, \dt
	\\&\qquad
	+
	2 \,
	\begin{pmatrix}
		\fint_0^{1/3} h_1 \, \dt
		\\
		-\fint_{3/4}^{1} h_1 \, \dt
	\end{pmatrix}^\top
	A_1
	\begin{pmatrix}
		\fint_0^{1/3} h_2 \, \dt
		\\
		-\fint_{3/4}^{1} h_2 \, \dt
	\end{pmatrix}
	\\
	g''(\bar x)[h_1,h_2]
	&:=
	2 \,
	\begin{pmatrix}
		\fint_0^{1/3} h_1 \, \dt
		\\
		-\fint_{3/4}^{1} h_1 \, \dt
	\end{pmatrix}^\top
	(A_2 - A_1)
	\begin{pmatrix}
		\fint_0^{1/3} h_2 \, \dt
		\\
		-\fint_{3/4}^{1} h_2 \, \dt
	\end{pmatrix}
	.
\end{align*}

Note that the quadratic functions $f$ and $g$
are
uniquely determined
via the first and second derivatives in $\bar x$
and the requirement $f(\bar x) = g(\bar x) = 0$.
Let us check that
the second-order sufficient condition \eqref{eq:SSC}
is satisfied.
For $h \in \TT_C(\bar x)$
we set
$\hat h := (\fint_0^{1/3} h_2 \, \dt, -\fint_{3/4}^{1} h_2 \, \dt) \ge 0$.
By utilizing \eqref{eq:crazy_matrix_property},
we have
\begin{align*}
	\sup_{(\lambda, \mu) \in \Lambda(\bar x)} \LL''(\bar x, \lambda, \mu)\,h^2
	&=
	\sup_{\mu \in [0,1]} f''(\bar x) \, h^2 + \mu \, g''(\bar x) \, h^2
	\\
	&=
	\int_0^{1/3} Z_0^{1/3}(h)^2 \, \dt
	+
	\int_{1/3}^{3/4} h^2 \, \dt
	+
	\int_{3/4}^1 Z_{3/4}^1(h)^2 \, \dt
	\\&\qquad\qquad
	+
	2\,\max( \hat h^\top A_1 \, \hat h, \hat h^\top A_2 \, \hat h )
	\\
	&\ge
	\int_0^{1/3} Z_0^{1/3}(h)^2 \, \dt
	+
	\int_{1/3}^{3/4} h^2 \, \dt
	+
	\int_{3/4}^1 Z_{3/4}^1(h)^2 \, \dt
	+
	\norm{\hat h}^2_{\R^2}
	\\
	&=
	\int_0^1 h^2 \, \dt
	=
	\norm{h}_{L^2(0,1)}^2.
\end{align*}
Hence, 
\cref{thm:sufficient_condition} implies
that $\bar x$ is a local minimizer.

It remains to check that the condition
\begin{equation}
	\label{eq:false_SNC}
	\LL''(\bar x, \lambda, \mu) \, h^2 \ge 0
	\qquad
	\forall h \in \KK(\bar x).
\end{equation}
is violated for all $(\lambda, \mu) \in \Lambda(\bar x)$.
To this end, we take
\begin{equation*}
	h_1 := \chi_{(0,1/3)},
	\qquad
	h_2 := -\chi_{(3/4,1)}
\end{equation*}
and observe $h_1, h_2 \in \KK(\bar x)$.
It is easy to check that
\begin{equation*}
	f''(\bar x) \, h_i^2 = 2 \, e_i^\top A_1 \, e_i
	,
	\quad
	g''(\bar x) \, h_i^2 = 2 \, e_i^\top (A_2 - A_1) \, e_i
\end{equation*}
for $i = 1, 2$.
Thus, for every $\mu \in [0,1]$ and the associated $\lambda$,
we have
\begin{equation*}
	\LL''(\bar x, \lambda, \mu) \, h_1^2 = 1 - 3 \, \mu,
	\qquad
	\LL''(\bar x, \lambda, \mu) \, h_2^2 = -1 + 2 \, \mu
\end{equation*}
and these two terms cannot be simultaneously non-negative.
Hence,
there does not exist any multiplier $(\lambda, \mu) \in \Lambda(\bar x)$
such that \eqref{eq:false_SNC} holds.
This means that \eqref{eq:false_SNC}
fails to be a necessary optimality condition.

\subsection{Constraint set which is not \texorpdfstring{$2$}{2}-polyhedric}
Next, we give a counterexample
to demonstrate that the assumption of $C$ being $(\hat m+1)$-polyhedric
in \cref{thm:SNC}
is crucial.
Therefore, we need a set which is polyhedric (i.e., $1$-polyhedric), but not $2$-polyhedric.
To the best of our knowledge,
the set given in \cite[Example~4.24]{Wachsmuth2016:2}
is the only known set with this property.
In order to state our counterexample,
we need to adapt the construction from \cite[Example~4.24]{Wachsmuth2016:2}.
In $\R^3$ we consider the points
\begin{align*}
	O   & := (0,0,0),  \\
	P_n & := (n^{-3}, n^{-2}, -n^{-4}), & n & \in \N, \\
	Q_n & := (-(n/\gamma)^{-3}, (n/\gamma)^{-2}, 0),      & n & \in \N,
\end{align*}
where $\gamma = (1 + \sqrt{3})/2$.
We set
\begin{equation*}
	C :=
	\conv
	\paren{
		\{O\}
		\cup
		\{P_n, Q_n\}_{n \in \N}
	}.
\end{equation*}
Since the sequences $\{P_n\}$ and $\{Q_n\}$ converge towards $O$,
the set $C$ is closed.
In what follows, we check that $C$ is polyhedric.
By arguing as in \cite[Example~4.24]{Wachsmuth2016:2},
we find that $C$ is polyhedric in $O$.
Next, it is a little bit tedious to check
that $C$ is the intersection of the half-spaces
which are defined by the following inequalities
and that the points on the right-hand side are exactly those points
of $O$, $P_n$, $Q_n$
which lie on the boundary of the half-spaces:
\begin{align*}
	x^\top (0, 0, 1)                    &\le 0,  && O, Q_n \quad \forall n \in \N, \\
	x^\top (1, \gamma, 1+\gamma)        &\ge 0,  && O, Q_1, P_1, \\
	x^\top (2\,k+1, -1, k\,(k+1)) &\le 0,  && O, P_k, P_{k+1}, \\
	x^\top ( a_k^{(1)}, b_k^{(1)}, c_k^{(1)}) & \le \gamma^3, && Q_k, Q_{k+1}, P_k, \\
	x^\top (a_k^{(2)}, b_k^{(2)}, c_k^{(2)}) & \le 1, && P_k, P_{k+1}, Q_{k+1},
\end{align*}
where $k \in \N$.
In the last two lines, we have used the coefficients
\begin{align*}
	a_k^{(1)} &:= k \, (k+1) \, (2\,k+1),
	&
	b_k^{(2)} &:= \frac{\gamma^3 \, \paren[\big]{(k+1)^4-k^4} + (k+1)^3 }{\gamma^3 \, (2\,k+1) + \gamma^2 \, (k+1)},
	\\
	b_k^{(1)} &:= \gamma \, (3\,k^2+3\,k+1),
	&
	a_k^{(2)} &:= (k+1)^4-k^4  - (2 \, k + 1) \, b_k^{(2)},
	\\
	c_k^{(1)} &:= k \, a_k^{(1)} + k^2 \, b_k^{(1)} - \gamma^3\,k^4,
	&
	c_k^{(2)} &:= k \, a_k^{(2)} + k^2 \, b_k^{(2)} - k^4.
\end{align*}
From this representation of $C$, we learn two things.
First, all $P_k$, $Q_k$ are extreme points of $C$ and, thus, $C$ is not polyhedral.
Second, 
the intersection $C \cap \{x \in \R^3 \mid x^\top (1,1,1) \ge \varepsilon\}$
is a polyhedron for all $\varepsilon > 0$,
since it can be written as a finite intersection of half-spaces.
Thus, $\RR_C(x)$ is closed
for all $x \in C \setminus \{O\}$.
Hence, $C$ is polyhedric at all $x \in C \setminus \set{O}$.

Hence, we have shown that $C$ is polyhedric, but not polyhedral.
As in \cite[Example~4.24]{Wachsmuth2016:2}, we can also check that $C$ is not $2$-polyhedric.

Next, we compute the intersection of $C$
with the hyperplane $x^\top (1,0,0) = 0$.
To this end, let $R_{k,n}$ be the intersection of this hyperplane
with the line segment joining $P_k$ and $Q_n$, i.e.,
\begin{equation*}
	R_{k,n}
	=
	(0, \lambda_{k,n} \, k^{-2} + (1-\lambda_{k,n}) \, (n/\gamma)^{-2}, -\lambda_{k,n} k^{-4}),
	\qquad
	\lambda_{k,n} = \frac1{1+\paren{\frac{n}{k\,\gamma}}^3}
	.
\end{equation*}
One can check that
\begin{equation*}
	C \cap (1,0,0)\anni
	=
	\conv\paren{ \{O\} \cup \{R_{k,n}\}_{k,n\in \N}}
	.
\end{equation*}
We define
\begin{equation*}
	\delta = \frac{\gamma^3 + 1}{\gamma \, (\gamma + 1)^2}
\end{equation*}
and claim that all points $R_{k,n}$ belong to the convex set
\begin{equation*}
	M
	:=
	\set{(0, x_2, x_3) \in \R^3 \given x_2 \ge 0 \text{ and } x_3 \le -\delta \, x_2^2 }
\end{equation*}
and that the points $R_{n,n}$ belong to the relative boundary of this set.
Indeed,
after a straightforward manipulation,
this claim is equivalent to the inequality
\begin{equation*}
	\frac{\gamma^3 \, k^3 + n^3}{\gamma^3 + 1}
	-
	\frac{k \, (\gamma \, k + n)^2}{(\gamma + 1)^2}
	\ge
	0
	\qquad\forall k,n \in \N
\end{equation*}
and that we have equality for $k = n$.
This latter equality is clear.
Moreover, one can check that
for $n \ge k$ the derivative w.r.t.\ $n$
and
for $k \ge n$ the derivative w.r.t.\ $k$
of the left-hand side is non-negative,
both by using the definition of $\gamma$.


Now, we consider the optimization problem
\begin{align*}
	\text{Minimize} \quad & f(x) = -\delta \, x_2^2 - x_3, \\
	\text{such that} \quad & x \in C, \\
	\text{and} \quad & x_1 = 0.
\end{align*}
In order to cast this problem in the form \eqref{eq:prob2},
we set $g(x) = x_1$ and $K = \{0\}$.
The feasible set of this problem is
$C \cap (1,0,0)\anni$
and this set is contained in $M$.
Hence,
\begin{equation*}
	-\delta \, x_2^2 - x_3 \ge 0
	\qquad\forall x \in M
\end{equation*}
shows that
$\bar x = (0,0,0)$ is a local minimizer
of the above problem.
Since $P_1 \in C$ has a positive $x_1$-coordinate
and since $Q_1 \in C$ has a negative $x_1$-coordinate,
it is easy to check that \eqref{eq:rzcq},
i.e.,
\begin{equation*}
	(1,0,0) \, \RR_C(\bar x)
	=
	\R^1
\end{equation*}
is satisfied.
Hence, there exist $\lambda \in \NN_C(\bar x)$, $\mu \in \R$
such that
the necessary condition from \cref{thm:fonc},
i.e.,
\begin{equation*}
	\begin{pmatrix}
		0 \\ 0 \\ -1
	\end{pmatrix}
	+
	\lambda
	+
	\mu \,
	\begin{pmatrix}
		1 \\ 0 \\ 0
	\end{pmatrix}
	=
	\begin{pmatrix}
		0 \\ 0 \\ 0
	\end{pmatrix}
\end{equation*}
is satisfied.
Finally, we check that the necessary optimality condition of second order \eqref{eq:snc}
does not hold.
Since the constraint $g$ is linear,
its second derivative vanishes and the precise value of the multiplier $\mu$
is irrelevant.
Next, we construct an element of
the tangent cone $\TT_C(\bar x)$.
From
\begin{equation*}
	n^2 \, (P_n - \bar x)
	=
	(n^{-1}, 1, -n^{-2})
	\to
	(0,1,0)
\end{equation*}
we find that $h := (0,1,0) \in \TT_C(\bar x)$.
Moreover,
$f'(\bar x) \, h = 0$
and
$g'(\bar x) \, h = 0$
are clear.
Thus, $h$ belongs to the critical cone $\KK(\bar x)$,
cf.\ \eqref{eq:critical}.
However,
\begin{equation*}
	\sup_{(\lambda, \mu) \in \Lambda(\bar x)} \LL''(\bar x, \lambda, \mu) \, h^2 \ge 0
	=
	f''(\bar x) \, h^2
	=
	- 2 \, \delta
\end{equation*}
is negative.
Hence, \eqref{eq:snc} is violated.

We mention that the only assumption of \cref{thm:SNC}
which does not hold
is the assumption that $C$ is $2$-polyhedric.
Hence, this assumption is essential.
On the other hand,
in the context of \cite{BonnansZidani1999},
the only assumption which might not hold
is the satisfaction of the regularity condition \eqref{eq:regularity_bonnans}.
We check that this condition indeed fails.
To this end, we start by computing the set of multipliers.
It is clear that $(\lambda, \mu)$ are multipliers at $\bar x$,
if and only if
the two conditions
\begin{equation*}
	\lambda
	=
	\mu \,
	\begin{pmatrix}
		-1 \\ 0 \\ 0
	\end{pmatrix}
	+
	\begin{pmatrix}
		0 \\ 0 \\ 1
	\end{pmatrix}
	\in
	\NN_C(\bar x)
\end{equation*}
hold.
Due to the construction of the set $C$
and due to $\bar x = (0,0,0)$,
we have
\begin{align*}
	\NN_C(\bar x)
	&=
	\set{ P_n, Q_n \given n \in \N }\polar
	\\
	&=
	\set*{
		\begin{pmatrix}
			n^{-3} \\ n^{-2} \\ -n^{-4}
		\end{pmatrix}
		,
		\begin{pmatrix}
			-(n/\gamma)^{-3} \\ (n/\gamma)^{-2} \\ 0
		\end{pmatrix}
	\given n \in \N }\polar
	=
	\set*{
		\begin{pmatrix}
			n \\ n^2 \\ -1
		\end{pmatrix}
		,
		\begin{pmatrix}
			-\gamma \\ 1 \\ 0
		\end{pmatrix}
		,
		\begin{pmatrix}
			0 \\ 1 \\ 0
		\end{pmatrix}
	\given n \in \N }\polar
	.
\end{align*}
Hence, $\mu \in \R$ has to satisfy the inequalities
\begin{align*}
	-n \, \mu - 1 &\le 0 \qquad \forall n \in \N,
	&
	\gamma \, \mu &\le 0,
	&
	0 &\le 0.
\end{align*}
Hence, $\mu = 0$ and $\lambda = (0,0,1)^\top$
are the unique Lagrange multipliers for $\bar x$.
Finally,
the regularity condition \eqref{eq:regularity_bonnans}
is violated, since
\begin{equation*}
	g'(\bar x) \, \paren[\big]{ \RR_C(\bar x) \cap \lambda\anni } - \RR_K(g(\bar x)) \cap \mu\anni
	=
	(-\infty, 0]
	\ne
	\R
	.
\end{equation*}
This example also shows that assuming \eqref{eq:regularity_bonnans} in
\cite{BonnansZidani1999}
cannot be replaced by the assumption of unique multipliers.

\section{Conclusions}
We have investigated problem \eqref{eq:prob}
featuring an abstract constraint $x \in C$
and
finitely many nonlinear constraints $g(x) \in K$.
Previously,
second-order necessary optimality conditions
have been obtained under the rather strong regularity condition
\eqref{eq:regularity_bonnans}.
We propose to use
the concept of $n$-polyhedricity of $C$
as a novel approach for deriving second-order necessary conditions.
In fact, ``almost all'' sets which are known to be polyhedric
are even $n$-polyhedric, see, e.g., \cite[Example~4.21]{Wachsmuth2016:2}.
This allows us to prove second-order necessary conditions
under the assumption of the CQ of Robinson, Zowe and Kurcyusz.
Second-order sufficient conditions can be obtained by the usual
contradiction argument.
By means of two counterexamples,
we have seen that the assumptions
and the formulation of \cref{thm:SNC}
is sharp.
The inclusion of the phenomenon of
two-norms discrepancy
is subject to future research.
It would also be interesting
to replace the finite-dimensional polyhedral cone $K$
by a set involving curvature,
e.g., the cone of semi-definite matrices.

\printbibliography
\end{document}